\documentclass{article}

\usepackage{arxiv}

\usepackage[utf8]{inputenc} 
\usepackage[T1]{fontenc}    
\usepackage{hyperref}       
\usepackage{url}            
\usepackage{booktabs}       
\usepackage{amsfonts}       
\usepackage{nicefrac}       
\usepackage{microtype}      
\usepackage{lipsum}
\usepackage{graphicx}
\graphicspath{ {./images/} }

\usepackage{amssymb}
\usepackage{amsmath}
\usepackage{amsthm}
\usepackage{xcolor}
\usepackage{subfig}
\usepackage{tikz}
\usetikzlibrary{arrows}

\newtheorem{theorem}{Theorem}
\newtheorem{lemma}{Lemma}
\newtheorem{cor}{Corollary}

\newtheorem{remark}{Remark}

\newtheorem*{theorem*}{Theorem}

\title{Billiard maps of confocal ellipses commute: a~geometric proof}

\author{
 Timur Bakiev \\
  Faculty of Mathematics\\
  HSE University\\
  Moscow, 109028, Russia \\
  \texttt{tnbakiev@edu.hse.ru} \\
   \And
 Ivan Molodyk \\
  Department of Mathematics\\
  Penn State University\\
  State College, PA 16801, USA \\
  \texttt{iam5268@psu.edu} \\
}

\usepackage[backend=biber,style=numeric-comp,sorting=none]{biblatex} 
\begin{document}

\definecolor{rvwvcq}{rgb}{0.08235294117647059,0.396078431372549,0.7529411764705882}
\definecolor{dtsfsf}{rgb}{0.8274509803921568,0.1843137254901961,0.1843137254901961}
\definecolor{sexdts}{rgb}{0.1803921568627451,0.49019607843137253,0.19607843137254902}
\definecolor{wrwrwr}{rgb}{0.3803921568627451,0.3803921568627451,0.3803921568627451}
\definecolor{dbwrru}{rgb}{0.8588235294117647,0.3803921568627451,0.0784313725490196}
\definecolor{ffcctt}{rgb}{1,0.8,0.2}
\definecolor{qqqqff}{rgb}{0,0,1}
\definecolor{ffxfqq}{rgb}{1,0.4980392156862745,0}
\definecolor{cczzff}{rgb}{0.8,0.6,1}
\definecolor{wwqqcc}{rgb}{0.4,0,0.8}
\definecolor{fuqqzz}{rgb}{0.9568627450980393,0,0.6}
\definecolor{ffttww}{rgb}{1,0.2,0.4}
\definecolor{ffdxqq}{rgb}{1,0.8431372549019608,0}
\definecolor{cxvqqq}{rgb}{0.7803921568627451,0.3137254901960784,0}

\maketitle

\begin{abstract}
We give a new, purely geometric proof of the classical fact that billiard maps in confocal ellipses commute. Existing proofs of this result rely on symplectic geometry and an invariant measure on the space of oriented lines; ours uses only elementary projective and Euclidean geometry. The argument rests on a main theorem describing how a billiard reflection can be constructed geometrically via tangents to confocal ellipses, from which the commutation property follows directly. This also yields, as a byproduct, an incidence result for tangents at four reflection points, revealing a hidden symmetry in the confocal billiard configuration. The main theorem recovers, via purely synthetic means, a fact previously established only by direct computation (Berman et al., 2024).
\end{abstract}


\section{Introduction}
\label{sec:intro}

A billiard system describes the motion of a free point inside a plane domain: the point moves with a constant speed along a straight line until it hits the~boundary, where it reflects according to the familiar law of geometrical optics: ''the~angle of incidence equals the angle of reflection``. The billiard map acts on oriented lines that intersect the billiard table, sending the~incoming billiard trajectory to the outgoing one. See Serge Tabachnikov's book, Geometry and Billiards~\cite{serge_2005}, for a detailed discussion of these standard definitions and classical results.

We will always use Latin uppercase letters ($A$, $B$, $C$) for points and two points to denote both lines and oriented lines, but notice that in the latter case the~order of letters matters; Latin lowercase letters ($a$, $b$, $c$) also denote lines; Gothic lowercase letters ($\mathfrak{a}$, $\mathfrak{b}$, $\mathfrak{c}$) will be used for ellipses, and Latin italic uppercase letters ($\mathcal{A}$, $\mathcal{B}$, $\mathcal{C}$) -- for billiard maps.

\begin{theorem}
    Let $\mathfrak{a}$ and $\mathfrak{b}$ be confocal ellipses. Denote $\mathcal{A}$ and $\mathcal{B}$ the billiard maps in $			\mathfrak{a}$ and $\mathfrak{b}$, respectively, acting on the~space of oriented lines intersecting both ellipses. Then $\mathcal{A}\mathcal{B} = \mathcal{B}\mathcal{A}$, i.e. they commute.
\end{theorem}

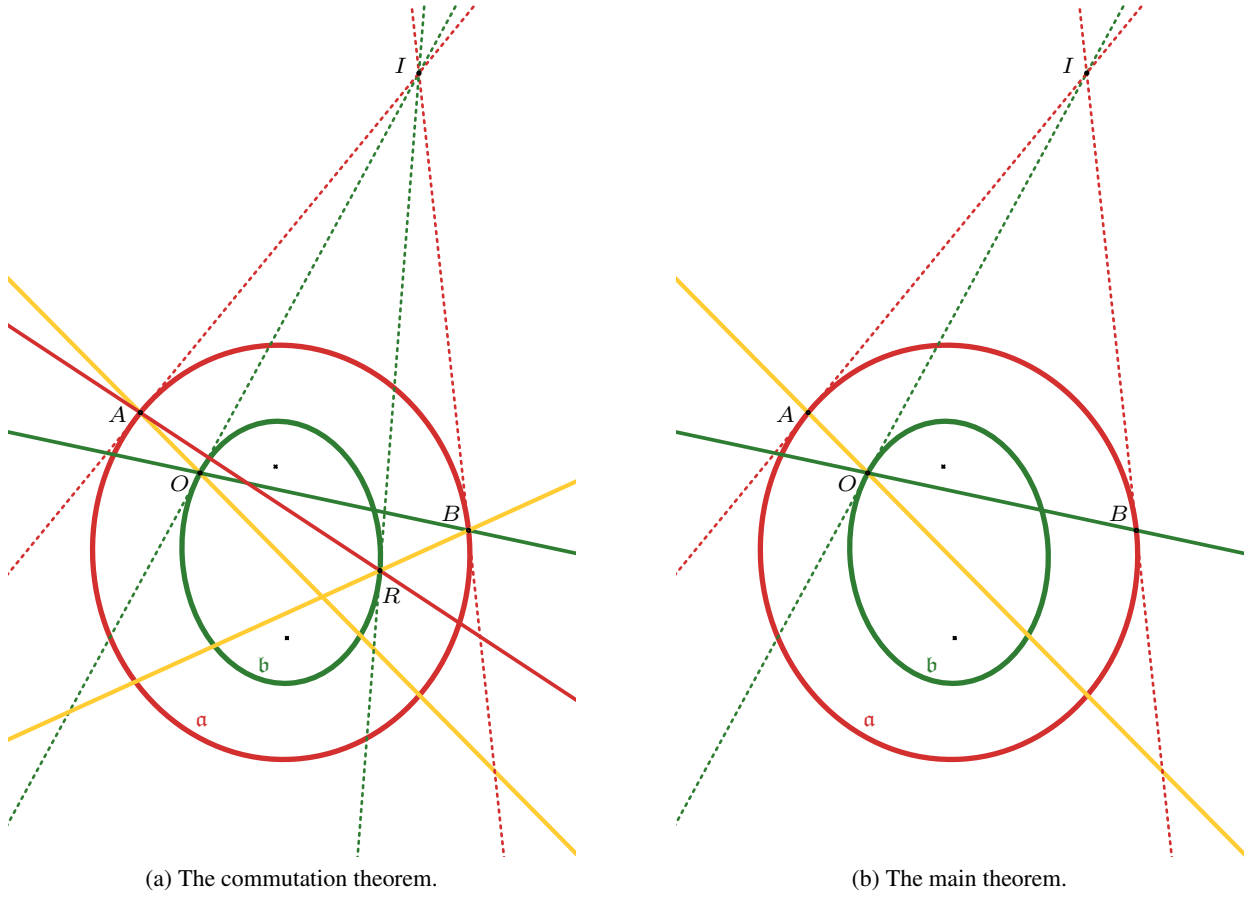
\begin{figure}[ht!]
	\centering
	
    \subfloat[The commutation theorem.]{\label{fig:teor1} \begin{tikzpicture}[line cap=round,line join=round,>=triangle 45,x=1cm,y=1cm,scale=0.3,every node/.style={scale=1.2},thick,rotate=90]
\clip(-15,-15) rectangle (25,10);
\draw [rotate around={3.7838926880849164:(-1.58,-2.02)},line width=2pt,color=sexdts] (-1.58,-2.02) ellipse (5.784225811185012cm and 4.371083187812707cm);
\draw [rotate around={3.7838926880849226:(-1.58,-2.02)},line width=2pt,color=dtsfsf] (-1.58,-2.02) ellipse (9.137991570470321cm and 8.315767549780759cm);
\draw [line width=1.6pt,color=ffcctt,domain=-21.08168864973216:22.49275790313854] plot(\x,{(-0.8811528282033123--2.606761587194037*\x)/2.6389309080196});
\draw [line width=1.0pt,dash pattern=on 1pt off 2pt,color=sexdts,domain=-21.08168864973216:22.49275790313854] plot(\x,{(--485.9674938934599-101.88646319839737*\x)/186.06282846069503});
\draw [line width=1.0pt,dash pattern=on 1pt off 2pt,color=dtsfsf,domain=-21.08168864973216:22.49275790313854] plot(\x,{(--651.9529915373092-67.29196402292307*\x)/81.99419117311172});
\draw [line width=1.0pt,dash pattern=on 1pt off 2pt,color=dtsfsf,domain=-21.08168864973216:22.49275790313854] plot(\x,{(--205.9496047962195-2.1869001553331717*\x)/-20.162089689635344});
\draw [line width=1.0pt,dash pattern=on 1pt off 2pt,color=sexdts,domain=-21.08168864973216:22.49275790313854] plot(\x,{(--143.95456802532706--1.7188773254683225*\x)/-21.937036593915032});
\draw [line width=1.4pt,color=sexdts,domain=-21.08168864973216:22.49275790313854] plot(\x,{(--18.770460129238497-11.842246342404781*\x)/-2.529707870320326});
\draw [line width=1.4pt,color=dtsfsf,domain=-21.08168864973216:22.49275790313854] plot(\x,{(--19.20489939264747-10.5666599068787*\x)/-6.967304277136593});
\draw [line width=1.6pt,color=ffcctt,domain=-21.08168864973216:22.49275790313854] plot(\x,
{(-20.635230765462687-3.9057774808014942*\x)/1.7749469042796866});
\begin{scriptsize}
\draw [color=black] (-5.36,-2.27)-- ++(-1.5pt,-1.5pt) -- ++(3pt,3pt) ++(-3pt,0) -- ++(3pt,-3pt);
\draw [color=black] (2.2,-1.77)-- ++(-1.5pt,-1.5pt) -- ++(3pt,3pt) ++(-3pt,0) -- ++(3pt,-3pt);
\draw[color=sexdts] (-6.5,-1.25) node {$\mathfrak{b}$};
\draw[color=dtsfsf] (-9.0,1.5) node {$\mathfrak{a}$};
\draw [fill=dbwrru] (1.9185545787412073,1.561261622300838) circle (2pt);
\draw[color=black] (1.45,2.45) node {$O$};
\draw [fill=dbwrru] (4.581204081277788,4.191452667576251) circle (2pt);
\draw[color=black] (4.5,5.2) node {$A$};
\draw [fill=black] (19.550936398056226,-8.094084564770771) circle (2pt);
\draw[color=black] (19.87330166952131,-7.284306881533164) node {$I$};
\draw [fill=dbwrru] (-2.3861001958588055,-6.375207239302449) circle (2pt);
\draw[color=black] (-3.5,-6.85) node {$R$};
\draw [fill=dbwrru] (-0.6111532915791189,-10.280984720103943) circle (2pt);
\draw[color=black] (0.15,-9.5) node {$B$};
\end{scriptsize}
\end{tikzpicture}}
    \hfill
    \subfloat[The main theorem.]{\label{fig:teor2} \begin{tikzpicture}[line cap=round,line join=round,>=triangle 45,x=1cm,y=1cm,scale=0.3,every node/.style={scale=1.2},thick,rotate=90]
\clip(-15,-15) rectangle (25,10);
\draw [rotate around={3.7838926880849164:(-1.58,-2.02)},line width=2pt,color=sexdts] (-1.58,-2.02) ellipse (5.784225811185012cm and 4.371083187812707cm);
\draw [rotate around={3.7838926880849226:(-1.58,-2.02)},line width=2pt,color=dtsfsf] (-1.58,-2.02) ellipse (9.137991570470321cm and 8.315767549780759cm);
\draw [line width=1.6pt,color=ffcctt,domain=-21.08168864973216:22.49275790313854] plot(\x,{(-0.8811528282033123--2.606761587194037*\x)/2.6389309080196});
\draw [line width=1.0pt,dash pattern=on 1pt off 2pt,color=sexdts,domain=-21.08168864973216:22.49275790313854] plot(\x,{(--485.9674938934599-101.88646319839737*\x)/186.06282846069503});
\draw [line width=1.0pt,dash pattern=on 1pt off 2pt,color=dtsfsf,domain=-21.08168864973216:22.49275790313854] plot(\x,{(--651.9529915373092-67.29196402292307*\x)/81.99419117311172});
\draw [line width=1.0pt,dash pattern=on 1pt off 2pt,color=dtsfsf,domain=-21.08168864973216:22.49275790313854] plot(\x,{(--205.9496047962195-2.1869001553331717*\x)/-20.162089689635344});
\draw [line width=1.4pt,color=sexdts,domain=-21.08168864973216:22.49275790313854] plot(\x,{(--18.770460129238497-11.842246342404781*\x)/-2.529707870320326});
\begin{scriptsize}
\draw [color=black] (-5.36,-2.27)-- ++(-1.5pt,-1.5pt) -- ++(3pt,3pt) ++(-3pt,0) -- ++(3pt,-3pt);
\draw [color=black] (2.2,-1.77)-- ++(-1.5pt,-1.5pt) -- ++(3pt,3pt) ++(-3pt,0) -- ++(3pt,-3pt);
\draw[color=sexdts] (-6.5,-1.25) node {$\mathfrak{b}$};
\draw[color=dtsfsf] (-9.0,1.5) node {$\mathfrak{a}$};
\draw [fill=dbwrru] (1.9185545787412073,1.561261622300838) circle (2pt);
\draw[color=black] (1.45,2.45) node {$O$};
\draw [fill=dbwrru] (4.581204081277788,4.191452667576251) circle (2pt);
\draw[color=black] (4.5,5.2) node {$A$};
\draw [fill=black] (19.550936398056226,-8.094084564770771) circle (2pt);
\draw[color=black] (19.87330166952131,-7.284306881533164) node {$I$};
\draw [fill=dbwrru] (-0.6111532915791189,-10.280984720103943) circle (2pt);
\draw[color=black] (0.15,-9.5) node {$B$};
\end{scriptsize}
\end{tikzpicture}}

    \caption{Main results.}
    \label{fig:main_res}
\end{figure}

Look at fig.\ref{fig:teor1}. Let $\mathfrak{a}$ be a red ellipse, and $\mathfrak{b}$ -- a green one. Choose the~yellow line $OA$ to start with. The respective reflection in the red ellipse is shown as the red line $AR$. Analogously, the green ellipse reflects $OA$ in the green line~$OB$.
Next, one should reflect the green line $OB$ off the red ellipse. The result corresponds to application of $\mathcal{A}\mathcal{B}$. To obtain the image of $OA$ under the~action of $\mathcal{B}\mathcal{A}$, it is left to reflect the red line $AR$ off the green ellipse. The~theorem says that both outcomes coincide and, in fact, equal to the yellow line $BR$.


\begin{remark}
    There are also tangents to ellipses at the reflection points $O$, $A$, $B$ and $R$ on fig.\ref{fig:teor1}. They are depicted both for convenience and to illustrate that they are incident. This fact, which we will prove together with the main theorem, is the first manifestation of a symmetry hidden deep inside the~system at hand.
\end{remark}

The commutation theorem is implied by the following theorem. 

\begin{theorem}[Main theorem]
    Let $\mathfrak{a}$ and $\mathfrak{b}$ be confocal ellipses (see fig.\ref{fig:teor2}). Take $A \in \mathfrak{a}$, $O \in \mathfrak{b}$. Denote by $I$ a point where the tangent line to $\mathfrak{a}$ at $A$ intersects with the tangent line to $\mathfrak{b}$ at $O$ ($I$ may lie on the line at infinity). If $B \neq A$ is a point, where the second tangent line from $I$ to $\mathfrak{a}$ touches the ellipse, then $OB = \mathcal{B}(OA)$.
\end{theorem}

\begin{proof}[Proof of the commutation theorem]
        Given the line $OA$, one can use the main theorem twice to find reflections $OB$ and $AR$. Then one can apply the~main theorem again to $RB$ to deduce that $BO$ and $RA$ are the respective reflections.
\end{proof}

The commutation theorem is a known fact, one can find it in~\cite{serge_2005}. However, the only proof existed so far is based on symplectic geometry and uses an~invariant measure. We are presenting a new, completely different proof. Our approach is purely geometrical and does not use any advanced mathematical concepts.

Moreover, the main theorem is essentially the same as the core Lemma $3.1$ in~\cite{berman2024}. In this paper we provide a purely geometrical proof of the statement, rather than a brute force calculation.

\section{Supplementary statements}
\label{sec:supp}

\subsection{Self-polar triangle lemma}
\label{subsec:supp-1}

\begin{lemma}[Dual self-polar triangle lemma]\label{lem::1}
	Let $\mathfrak{a}$ and $\mathfrak{b}$ be confocal ellipses (see fig.\ref{fig:lemma1}). Suppose $q$ is tangent to $\mathfrak{b}$ at~$O$, $p$ is normal to $\mathfrak{b}$ at $O$ and $o$ is the polar of $O$ w.r.t. $\mathfrak{a}$. Denote by $Q$ the intersection point of $o$ and $p$. Then $Q$ is the~pole of $q$ w.r.t. $\mathfrak{a}$. 
\end{lemma}

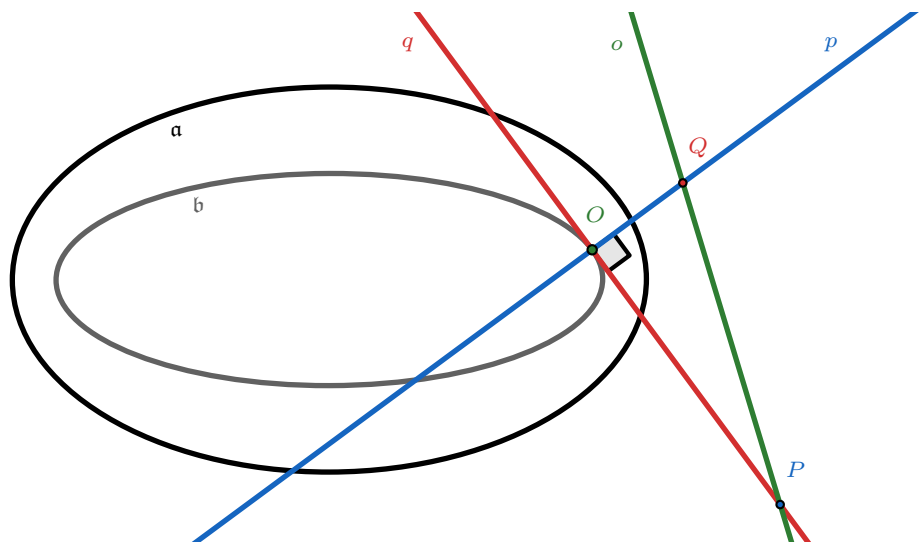
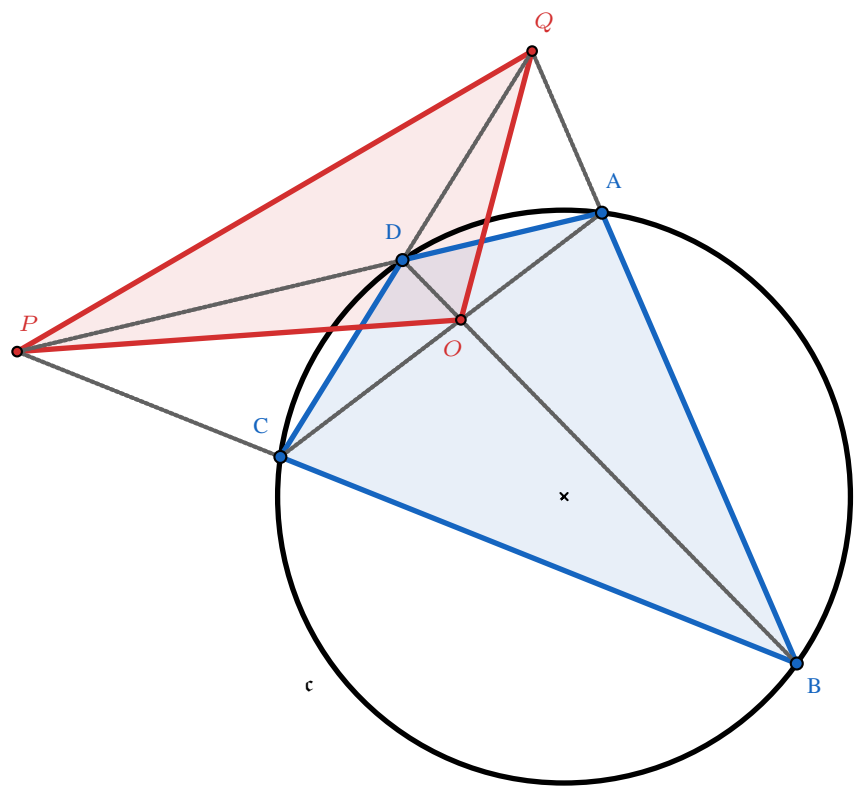
\begin{figure}[ht!]
    \centering
    
    \subfloat[Dual self-polar triangle lemma.]{\label{fig:lemma1} \begin{tikzpicture}[line cap=round,line join=round,>=triangle 45,x=1cm,y=1cm,scale=0.7,every node/.style={scale=1.2},thick]
\clip(-10,-5) rectangle (10,5);
\draw[line width=1.6pt,fill=black,fill opacity=0.10000000149011612] (2.6848745634359688,0.11772297792343572) -- (3.091854147096124,0.41804424341869595) -- (2.7915328816008635,0.825023827078851) -- (2.3845532979407085,0.5247025615835907) -- cycle; 
\draw [rotate around={0.15197783856389313:(-2.5819374693180546,-0.037533764331688375)},line width=2pt,color=wrwrwr] (-2.5819374693180546,-0.037533764331688375) ellipse (5.165498363004727cm and 2.0046243253924185cm);
\draw [rotate around={0.1519778385638931:(-2.581937469318054,-0.03753376433168837)},line width=2pt] (-2.581937469318054,-0.03753376433168837) ellipse (5.991773741253943cm and 3.638337245793851cm);
\draw [line width=2pt,color=dtsfsf,domain=-15.147432714335514:9.935665661675777] plot(\x,{(--883.6293096980512-318.7988569211707*\x)/235.25031719762393});
\draw [line width=2pt,color=rvwvcq,domain=-15.147432714335514:9.935665661675777] plot(\x,{(-393.69234285873307--235.25031719762393*\x)/318.7988569211707});
\draw [line width=2pt,color=sexdts,domain=-15.147432714335514:9.935665661675777] plot(\x,{(--487.7379584570548-105.13744604457356*\x)/31.8182237356645});
\begin{scriptsize}
\draw[color=wrwrwr] (-5.059474073746713,1.3790932673967506) node {$\mathfrak{b}$};
\draw[color=black] (-5.476260386380558,2.768380976176232) node {$\mathfrak{a}$};
\draw [fill=sexdts] (2.384553297940707,0.5247025615835886) circle (2.5pt);
\draw[color=sexdts] (2.4510727099794007,1.1848966892382896) node {$O$};
\draw[color=dtsfsf] (-1.1,4.4) node {$q$};
\draw[color=rvwvcq] (6.9,4.4) node {$p$};
\draw[color=sexdts] (2.85,4.4) node {$o$};
\draw [fill=dtsfsf] (4.097678723223334,1.7888642451893262) circle (2pt);
\draw[color=dtsfsf] (4.387682345953751,2.4557912417008487) node {$Q$};
\draw [fill=rvwvcq] (5.937283726660525,-4.289770860245994) circle (2pt);
\draw[color=rvwvcq] (6.228488560086562,-3.6223424842093817) node {$P$};
\end{scriptsize}
\end{tikzpicture}}
	\vfill
    \subfloat[Brocard's theorem.]{\label{fig:BT} \begin{tikzpicture}[line cap=round,line join=round,>=triangle 45,x=1cm,y=1cm,scale=0.9,every node/.style={scale=1.2},thick]
\clip(-12,-6) rectangle (4,8);
\fill[line width=2pt,color=rvwvcq,fill=rvwvcq,fill opacity=0.10000000149011612] (-5.37455172011361,2.406342859133738) -- (-2.4423812865859875,3.1028308036865253) -- (0.42109115928762053,-3.523485444340097) -- (-7.169377964857745,-0.4871252862721501) -- cycle;
\fill[line width=2pt,color=dtsfsf,fill=dtsfsf,fill opacity=0.1] (-11.039057455980537,1.0608344562899843) -- (-4.516051018324804,1.527965452873441) -- (-3.4689491630689986,5.478395272429388) -- cycle;
\draw [line width=2pt] (-3,-1.07) circle (4.209923460794108cm);
\draw [line width=2pt,color=rvwvcq] (-5.37455172011361,2.406342859133738)-- (-2.4423812865859875,3.1028308036865253);
\draw [line width=2pt,color=rvwvcq] (-2.4423812865859875,3.1028308036865253)-- (0.42109115928762053,-3.523485444340097);
\draw [line width=2pt,color=rvwvcq] (0.42109115928762053,-3.523485444340097)-- (-7.169377964857745,-0.4871252862721501);
\draw [line width=2pt,color=rvwvcq] (-7.169377964857745,-0.4871252862721501)-- (-5.37455172011361,2.406342859133738);
\draw [line width=1.6pt,dash pattern=on 1pt off 1pt,color=wrwrwr] (-5.37455172011361,2.406342859133738)-- (0.42109115928762053,-3.523485444340097);
\draw [line width=1.6pt,dash pattern=on 1pt off 1pt,color=wrwrwr] (-2.4423812865859875,3.1028308036865253)-- (-7.169377964857745,-0.4871252862721501);
\draw [line width=2pt,color=dtsfsf] (-11.039057455980537,1.0608344562899843)-- (-4.516051018324804,1.527965452873441);
\draw [line width=2pt,color=dtsfsf] (-4.516051018324804,1.527965452873441)-- (-3.4689491630689986,5.478395272429388);
\draw [line width=2pt,color=dtsfsf] (-3.4689491630689986,5.478395272429388)-- (-11.039057455980537,1.0608344562899843);
\draw [line width=1.6pt,dash pattern=on 1pt off 1pt,color=wrwrwr] (-11.039057455980537,1.0608344562899843)-- (-7.169377964857745,-0.4871252862721501);
\draw [line width=1.6pt,dash pattern=on 1pt off 1pt,color=wrwrwr] (-5.37455172011361,2.406342859133738)-- (-11.039057455980537,1.0608344562899843);
\draw [line width=1.6pt,dash pattern=on 1pt off 1pt,color=wrwrwr] (-3.4689491630689986,5.478395272429388)-- (-2.4423812865859875,3.1028308036865253);
\draw [line width=1.6pt,dash pattern=on 1pt off 1pt,color=wrwrwr] (-3.4689491630689986,5.478395272429388)-- (-5.37455172011361,2.406342859133738);
\begin{scriptsize}
\draw [color=black] (-3,-1.07)-- ++(-1.5pt,-1.5pt) -- ++(3pt,3pt) ++(-3pt,0) -- ++(3pt,-3pt);
\draw [fill=rvwvcq] (-2.4423812865859875,3.1028308036865253) circle (2.5pt);
\draw[color=rvwvcq] (-2.2681657174153687,3.57103514583257) node {A};
\draw [fill=rvwvcq] (-5.37455172011361,2.406342859133738) circle (2.5pt);
\draw[color=rvwvcq] (-5.508053447169597,2.8241728691500867) node {D};
\draw [fill=rvwvcq] (-7.169377964857745,-0.4871252862721501) circle (2.5pt);
\draw[color=rvwvcq] (-7.453763031168461,-0.022160968311485563) node {C};
\draw [fill=rvwvcq] (0.42109115928762053,-3.523485444340097) circle (2.5pt);
\draw[color=rvwvcq] (0.6846142277535485,-3.8491564826510236) node {B};
\draw [fill=dtsfsf] (-4.516051018324804,1.527965452873441) circle (2pt);
\draw[color=dtsfsf] (-4.6336975601316492,1.105411310043271) node {$O$};
\draw [fill=dtsfsf] (-3.4689491630689986,5.478395272429388) circle (2pt);
\draw[color=dtsfsf] (-3.2916821862927663,5.901168383489625) node {$Q$};
\draw [fill=dtsfsf] (-11.039057455980537,1.0608344562899843) circle (2pt);
\draw[color=dtsfsf] (-10.870059445214775,1.4804483157851198) node {$P$};
\draw[color=black] (-6.75,-3.8491564826510236) node {$\mathfrak{c}$};
\end{scriptsize}
\end{tikzpicture}}

    \caption{Supplementary statements.}
    \label{fig:supp_st}
\end{figure}

    \begin{proof}
        Consider the case when $\mathfrak{b}$ is given by an equation
        $$
            \frac{x^2}{a^2} + \frac{y^2}{b^2} = 1,
        $$
        and $\mathfrak{a}$ is given by
        $$
            \frac{x^2}{a^2 + \lambda} + \frac{y^2}{b^2 + \lambda} = 1, \quad \lambda \in \mathbb{R}.
        $$
        The general case can be reduced to this one with application of an appropriate isometry.
        
        Let $O$ have coordinates $(x_0,\ y_0)$. Then we can write equations of $q$ and $o$:
        $$
            \begin{aligned}
                &q: \quad \frac{xx_0}{a^2} + \frac{yy_0}{b^2} = 1;\\
                &o: \quad \frac{xx_0}{a^2 + \lambda} + \frac{yy_0}{b^2 + \lambda} = 1.
            \end{aligned}
        $$
        The third line, $p$, we can represent with its parametric equations:
        $$
            \left\{
            \begin{aligned}
                x =& x_0 + \frac{x_0}{a^2}t =  x_0 \left(\frac{a^2 + t}{a^2}\right)\\
                y =& y_0 + \frac{y_0}{b^2}t = y_0 \left(\frac{b^2 + t}{b^2}\right)
            \end{aligned}
            \right. , \quad t \in \mathbb{R}.
        $$

        The pole $F$ of $q$ w.r.t. $\mathfrak{a}$ belongs to $o$ and has coordinates
        $$
            \left( x_0 \left(\frac{a^2 + \lambda}{a^2}\right),\ y_0 \left(\frac{b^2 + \lambda}{b^2}\right) \right),
        $$
        since it follows from the equation of $q$ that
        $$
            \frac{xx_0}{a^2 + \lambda}\left(\frac{a^2 + \lambda}{a^2}\right) + \frac{yy_0}{b^2 + \lambda}\left(\frac{b^2 + \lambda}{b^2}\right) = 1.
        $$
        But now we see that $F$ lies on $p$. Hence $F = Q$.
    \end{proof}

\begin{cor}[Self-polar triangle lemma]\label{SPTL}
    Let OPQ be a self-polar triangle with respect to an ellipse $\mathfrak{a}$. Let there exist an ellipse confocal to $\mathfrak{a}$ that is tangent to the line OP at O. Then the angle O in the triangle is right. 
\end{cor}

\begin{proof}
    Polar reciprocation is a bijection, so there exists only one self-polar triangle with a side on the line $OP$ and a~vertex $O$. By the previous lemma, $\angle O$ is right in this triangle.
\end{proof}

\subsection{Brocard's theorem}
\label{subsec:supp-2}

\begin{theorem*}[Brocard's theorem]
    Let $ABCD$ be a complete quadrilateral inscribed into a conic $\mathfrak{c}$. Denote $O = AC \cap BD$, $P = AD \cap BC$ and $Q = BA \cap CD$, like on Figure \ref{fig:BT}. Then $\triangle OPQ$ is self-polar.
\end{theorem*}

\begin{proof}
    A proof for the case when $\mathfrak{c}$ is a circle can be found in~\cite{chen_2016} . However, the whole statement is formulated in purely projective terms and the proof uses only projective methods. Therefore, the~theorem holds for an arbitrary conic.
\end{proof}

\section{Proof of the main theorem}

\subsection{The primary configuration}

Let us start from the same configuration as the one in fig.\ref{fig:teor2}. First of all, we add points $C$ and $D$. The former one, $C$, is the intersection of $OA$ and $\mathfrak{a}$ different from $A$. The latter one, $D$, is the intersection of $OB$ and $\mathfrak{a}$ different from $B$.

\begin{figure}[ht!]
    \centering
    \begin{tikzpicture}[line cap=round,line join=round,>=triangle 45,x=1cm,y=1cm,scale=0.35,every node/.style={scale=1.2},thick,rotate=90]
\clip(-30,-10) rectangle (20,15);
\fill[line width=1.6pt,color=rvwvcq,fill=rvwvcq,fill opacity=0.10000000149011612] (-11.506274447981097,12.473172474350017) -- (-28.758756416458418,-5.145876418019076) -- (-8.057085818678102,1.1729143743030448) -- cycle;
\draw [rotate around={0:(-2.8675950000000032,-1.3)},line width=2pt,color=sexdts] (-2.8675950000000032,-1.3) ellipse (8.305114854748815cm and 3.167394073452877cm);
\draw [rotate around={0:(-2.8675949999999997,-1.3)},line width=2pt,color=dtsfsf] (-2.8675949999999997,-1.3) ellipse (10.715364490448604cm and 7.475057767612359cm);
\draw [line width=2pt,color=ffcctt,domain=-33.16393378101358:25.28334327839499] plot(\x,{(--22.28227067006179--2.529315473789505*\x)/1.622760260193992});
\draw [line width=1.6pt,color=sexdts,domain=-33.16393378101358:25.28334327839499] plot(\x,{(--991.2620438633114--83.30075355311851*\x)/272.9105642647503});
\draw [line width=1.2pt,color=dtsfsf,domain=-33.16393378101358:25.28334327839499] plot(\x,{(--904.4362895410949--18.511685821984865*\x)/134.73635318187326});
\draw [line width=1.2pt,color=dtsfsf,domain=-33.16393378101358:25.28334327839499] plot(\x,{(--151.99552941703567-14.30244002636461*\x)/-11.966914193868545});
\draw [line width=2pt,color=sexdts,domain=-33.16393378101358:25.28334327839499] plot(\x,{(-33.34361026721655-6.241080373825496*\x)/14.44377379723882});
\draw [line width=1.5pt,dash pattern=on 4pt off 4pt,color=wwqqcc,domain=-33.16393378101358:25.28334327839499] plot(\x,{(-57.52589369274979-3.2327381067159537*\x)/-6.887774053878548});
\draw [line width=1.5pt,dash pattern=on 4pt off 4pt,color=wwqqcc,domain=-33.16393378101358:25.28334327839499] plot(\x,{(--93.89088714914055-0.04084833961340362*\x)/-18.474138684172885});
\draw [line width=1.5pt,dash pattern=on 4pt off 4pt,color=cczzff,domain=-30.0:5.0] plot(\x,{(-94.28121879918416-7.91044516055839*\x)/-0.26147844766857276});
\draw [line width=1.5pt,dash pattern=on 4pt off 4pt,color=cczzff,domain=-33.16393378101358:25.28334327839499] plot(\x,{(--13.51074591720138-11.10233492766094*\x)/11.324886182625766});
\draw [line width=1.6pt,color=rvwvcq] (-11.506274447981097,12.473172474350017)-- (-28.758756416458418,-5.145876418019076);
\draw [line width=1.6pt,color=rvwvcq] (-28.758756416458418,-5.145876418019076)-- (-8.057085818678102,1.1729143743030448);
\draw [line width=1.6pt,color=rvwvcq] (-8.057085818678102,1.1729143743030448)-- (-11.506274447981097,12.473172474350017);
\begin{scriptsize}
\draw [color=black] (-10.545,-1.3)-- ++(-1.5pt,-1.5pt) -- ++(3pt,3pt) ++(-3pt,0) -- ++(3pt,-3pt);
\draw [color=black] (4.80981,-1.3)-- ++(-1.5pt,-1.5pt) -- ++(3pt,3pt) ++(-3pt,0) -- ++(3pt,-3pt);
\draw [fill=rvwvcq] (-8.057085818678102,1.1729143743030448) circle (4pt);
\draw[color=rvwvcq] (-7.5,0.0) node {$O$};
\draw [fill=dtsfsf] (-4.938198204065049,6.034168928138489) circle (4pt);
\draw[color=dtsfsf] (-5.0,7.5) node {$A$};
\draw [fill=black] (18.353602172429262,9.23427402684216) circle (4pt);
\draw[color=black] (18.5,10.75) node {$I$};
\draw [fill=dtsfsf] (6.386687978560717,-5.068165999522451) circle (4pt);
\draw[color=dtsfsf] (6.4,-3.7) node {$B$};
\draw [fill=dtsfsf] (-11.825972257943597,2.801430821422535) circle (4pt);
\draw[color=dtsfsf] (-12.5,4.0) node {$D$};
\draw [fill=dtsfsf] (-12.08745070561217,-5.109014339135855) circle (4pt);
\draw[color=dtsfsf] (-13.2,-5.7) node {$C$};
\draw [fill=rvwvcq] (-11.506274447981097,12.473172474350017) circle (4pt);
\draw[color=rvwvcq] (-10.5,13.0) node {$Q$};
\draw [fill=rvwvcq] (-28.758756416458418,-5.145876418019076) circle (4pt);
\draw[color=rvwvcq] (-28.8,-3.7) node {$P$};
\end{scriptsize}
\end{tikzpicture}
    \caption{The primary configuration.}
    \label{fig:dual-quad}
\end{figure}
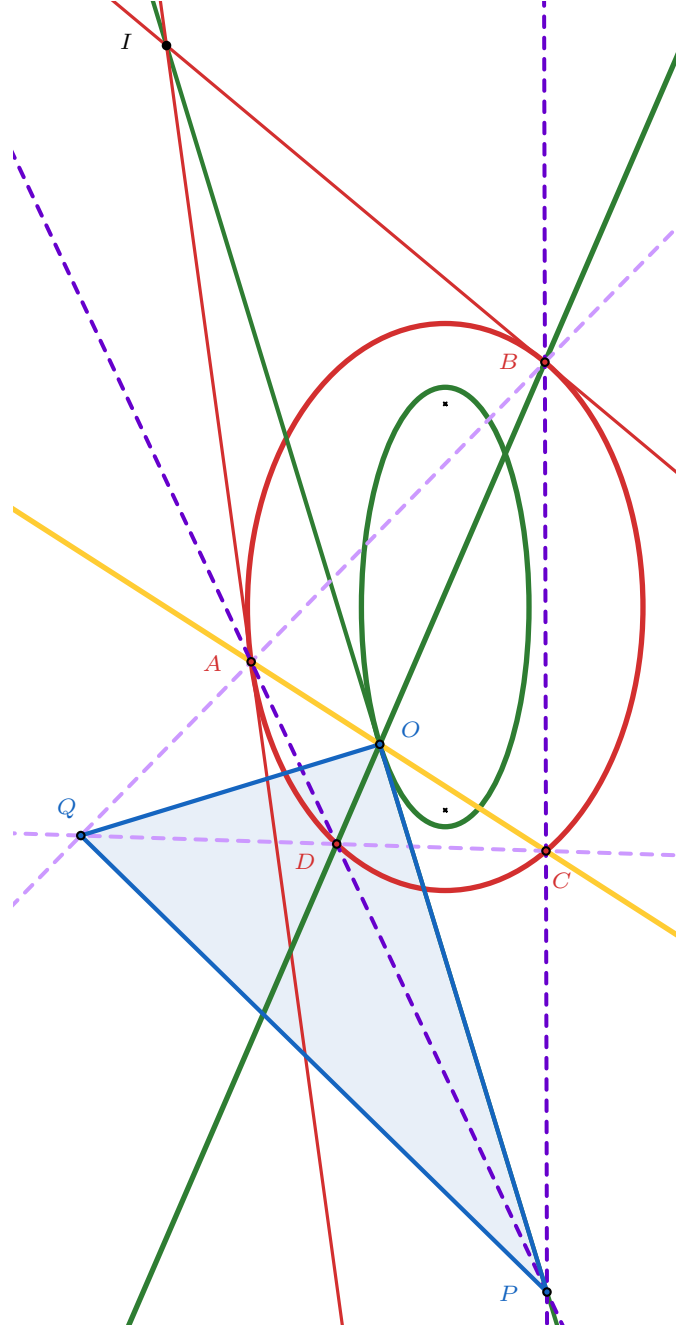

Denote $Q = CD \cap BA$, $P = AD \cap BC$. By Brocard's theorem applied to $ABCD$ and $\mathfrak{a}$, $\triangle OPQ$ is self-polar with respect to the ellipse $\mathfrak{a}$.

Next, we need the well-known Pascal's theorem.

\begin{theorem*}[Pascal's theorem]
    Let six points $A, B, C, D, E, F$ lie on a conic on a projective plane. Then three crossing points $P = AB\cap DE;\, Q = BC\cap EF;\, R = CD\cap FA$ are collinear.
    
\end{theorem*}

A simple proof of Pascal's theorem was published by van~Yzeren~\cite{van_Yzeren_1993}.

Let us apply Pascal's theorem to the~degenerate hexagon $AACBBD$\footnote{Pascal's theorem for this degenerate hexagon is a limit case for non-degenerate hexagons, where two distinct points $A_1, A_2$ on the conic approach each other and merge at a point $A$, and analogously for $B$.}, where $AA$ will denote the tangent line to the~corresponding conic through point $A$. The~theorem proves that $P$, $O$ and $I$ are incident. In other words, the side $OP$ of~$\triangle OPQ$ lies on the tangent line $IO$ to the ellipse $\mathfrak{b}$. Hence $\angle O$ is right by the Self-polar triangle lemma.

\subsection{Construction of a projective transformation}

One can view the Euclidean plane as an affine chart on $\mathbb{RP}^2$. Without loss of generality we may assume that in~homogeneous coordinates on $\mathbb{RP}^2$
\begin{align*}
    O &= [1 : 0 : 0], \\
    P &= [1 : p : 0], \\
    Q &= [1 : 0 : q].
\end{align*}

The space of lines on $\mathbb{RP}^2$ forms the dual projective plane $(\mathbb{RP}^2)^*$, and it has dual coordinates: the line given by the~equation $ax+by+cz = 0$ will~have coordinates $(a:b:c)$. Then $OP = (0 : 0 : 1)$, $OQ = (0 : 1: 0)$ and the~line at~infinity $\infty = (1 : 0 : 0)$ in the dual coordinates.

\begin{figure}[ht!]
    \centering
    
    \subfloat[The original plane.]{\label{fig::before} \begin{tikzpicture}[line cap=round,line join=round,>=triangle 45,x=1cm,y=1cm,scale=0.225,every node/.style={scale=1.2},thick,rotate=90]
    \clip(-30,-15) rectangle (20,15);
    \draw[line width=0.8pt,fill=black,fill opacity=0.25] (-8.439822280119717,2.426838489048527) -- (-9.693746394865200,2.04410202760691) -- (-9.311009933423584,0.790177912861429) -- (-8.057085818678102,1.1729143743030448) -- cycle;
    \draw [rotate around={0:(-2.8675949999999997,-1.3)},line width=2pt,color=dtsfsf] (-2.8675949999999997,-1.3) ellipse (10.715364490448604cm and 7.475057767612359cm);
    \draw [line width=0.8pt,color=ffcctt,domain=-33.85882362578984:22.999738231281484] plot(\x,{(--22.28227067006179--2.529315473789505*\x)/1.622760260193992});
    \draw [line width=1.2pt,color=sexdts,domain=-33.85882362578984:22.999738231281484] plot(\x,{(--991.2620438633114--83.30075355311851*\x)/272.9105642647503});
    \draw [line width=0.8pt,color=sexdts,domain=-33.85882362578984:22.999738231281484] plot(\x,{(-33.34361026721655-6.241080373825496*\x)/14.44377379723882});
    \draw [line width=1.2pt,color=rvwvcq,domain=-33.85882362578984:22.999738231281484] plot(\x,{(-87.00154636229871-11.300258100046973*\x)/3.449188629302995});
    \draw [line width=0.4pt,color=fuqqzz,domain=-33.85882362578984:22.999738231281484] plot(\x,{(-221.68649235290266-3.4696802873675585*\x)/23.689415023386537});
    \draw [line width=0.4pt,color=fuqqzz,domain=-33.85882362578984:22.999738231281484] plot(\x,{(--183.842309585696--9.901118486413187*\x)/19.608233276304908});
    \begin{scriptsize}
    \draw [fill=rvwvcq] (-8.057085818678102,1.1729143743030448) circle (4pt);
    \draw[color=rvwvcq] (-7.8,3.0) node {$O$};
    \draw [fill=dtsfsf] (-4.938198204065049,6.034168928138489) circle (4pt);
    \draw[color=dtsfsf] (-4.25,5.1) node {$A$};
    \draw [fill=black] (18.353602172429262,9.23427402684216) circle (4pt);
    \draw[color=black] (18.75,8.5) node {$I$};
    \draw [fill=dtsfsf] (6.386687978560717,-5.068165999522451) circle (4pt);
    \draw[color=dtsfsf] (6.0,-3.75) node {$B$};
    \draw [fill=dtsfsf] (-11.825972257943597,2.801430821422535) circle (4pt);
    \draw[color=dtsfsf] (-11.25,1.7) node {$D$};
    \draw [fill=dtsfsf] (-12.08745070561217,-5.109014339135855) circle (4pt);
    \draw[color=dtsfsf] (-11.0,-5.0) node {$C$};
    \draw [fill=rvwvcq] (-11.506274447981097,12.473172474350017) circle (4pt);
    \draw[color=rvwvcq] (-10.75,13.5) node {$Q$};
    \draw [fill=rvwvcq] (-28.758756416458418,-5.145876418019076) circle (4pt);
    \draw[color=rvwvcq] (-28.25,-6.1) node {$P$};
    \draw [fill=fuqqzz] (-9.150523140153508,4.7552420683941135) circle (4pt);
    \draw[color=fuqqzz] (-8.65,6.15) node {$U$};
    \draw [fill=fuqqzz] (-5.069341393071884,-8.615556705386632) circle (4pt);
    \draw[color=fuqqzz] (-4.4,-7.8) node {$V$};
    \end{scriptsize}
\end{tikzpicture}}
    \hfill
    \subfloat[The image under $\pi$.]{\label{fig::after} \begin{tikzpicture}[line cap=round,line join=round,>=triangle 45,x=1cm,y=1cm,scale=0.95,every node/.style={scale=1.2},thick,rotate=90]
\clip(-7,-2) rectangle (5,4);
\draw[line width=0.8pt,fill=black,fill opacity=0.25] (-1.4723955391833172,1.4671315739360025) -- (-1.679809672512561,1.4683962942611808) -- (-1.6810743928377394,1.2609821609319372) -- (-1.4736602595084956,1.2597174406067588) -- cycle; 
\draw [rotate around={-0.34936017950278525:(-1.48,0.22)},line width=2pt,color=dtsfsf] (-1.48,0.22) ellipse (3.65803102696984cm and 1.6193798177926113cm);
\draw [line width=1.2pt,color=rvwvcq,domain=-2.0:3.5] plot(-1.4736602595084956,\x);
\draw [line width=1.2pt,color=sexdts,domain=-7.0:5.0] plot(\x,{(-8.2048--0.04*\x)/-6.56});
\draw [line width=0.8pt,color=ffcctt,domain=-7.0:2.0] plot(\x,{(-6.700183135466573-1.8398372449129843*\x)/-3.1664943066545232});
\draw [line width=0.8pt,color=sexdts,domain=-5.0:4.0] plot(\x,{(--1.1923303628420463-1.8783147873189343*\x)/3.1438226479214024});
\draw [line width=0.4pt,color=fuqqzz,domain=-7.0:5.0] plot(\x,{(--634.5969426610379-2.114031981627683*\x)/346.7012449869413});
\draw [line width=0.4pt,color=fuqqzz,domain=-7.0:5.0] plot(\x,{(--488.3059295324017--2.114031981627697*\x)/-346.70124498694116});
\begin{scriptsize}
\draw [fill=rvwvcq] (-1.4736602595084956,1.259717440606759) circle (1.2pt);
\draw[color=rvwvcq] (-1.2,0.8) node {$O'$};
\draw [fill=dtsfsf] (-2.369798078285086,1.7951257969263188) circle (1.2pt);
\draw[color=dtsfsf] (-2.25,2.1) node {$D'$};
\draw [fill=black] (4.1449683347709145,1.2254575101538356) circle (1.2pt);
\draw[color=black] (4.265,0.95) node {$I'$};
\draw [fill=dtsfsf] (-0.5710599480659858,1.7841578814981536) circle (1.2pt);
\draw[color=dtsfsf] (-0.7,2.1) node {$A'$};
\draw [fill=dtsfsf] (-4.640154566163019,-0.5801198043062252) circle (1.2pt);
\draw[color=dtsfsf] (-4.7,-0.25) node {$C'$};
\draw [fill=dtsfsf] (1.6701623884129069,-0.6185973467121757) circle (1.2pt);
\draw[color=dtsfsf] (1.8,-1.0) node {$B'$};
\draw [fill=rvwvcq] (-1.464621211621663,2.7421212940472737) circle (1.2pt);
\draw[color=rvwvcq] (-1.2,2.9) node {$Q'$};
\draw [fill=fuqqzz] (-1.47012591637708,1.8393497141589394) circle (1.2pt);
\draw[color=fuqqzz] (-1.2,2.1) node {$U'$};
\draw [fill=fuqqzz] (-1.4898740836229207,-1.399349714158939) circle (1.2pt);
\draw[color=fuqqzz] (-1.2,-1.1) node {$V'$};
\end{scriptsize}
\end{tikzpicture}}
    
    \caption{How the transformation $\pi$ works.}
    \label{fig:before-after}
\end{figure}
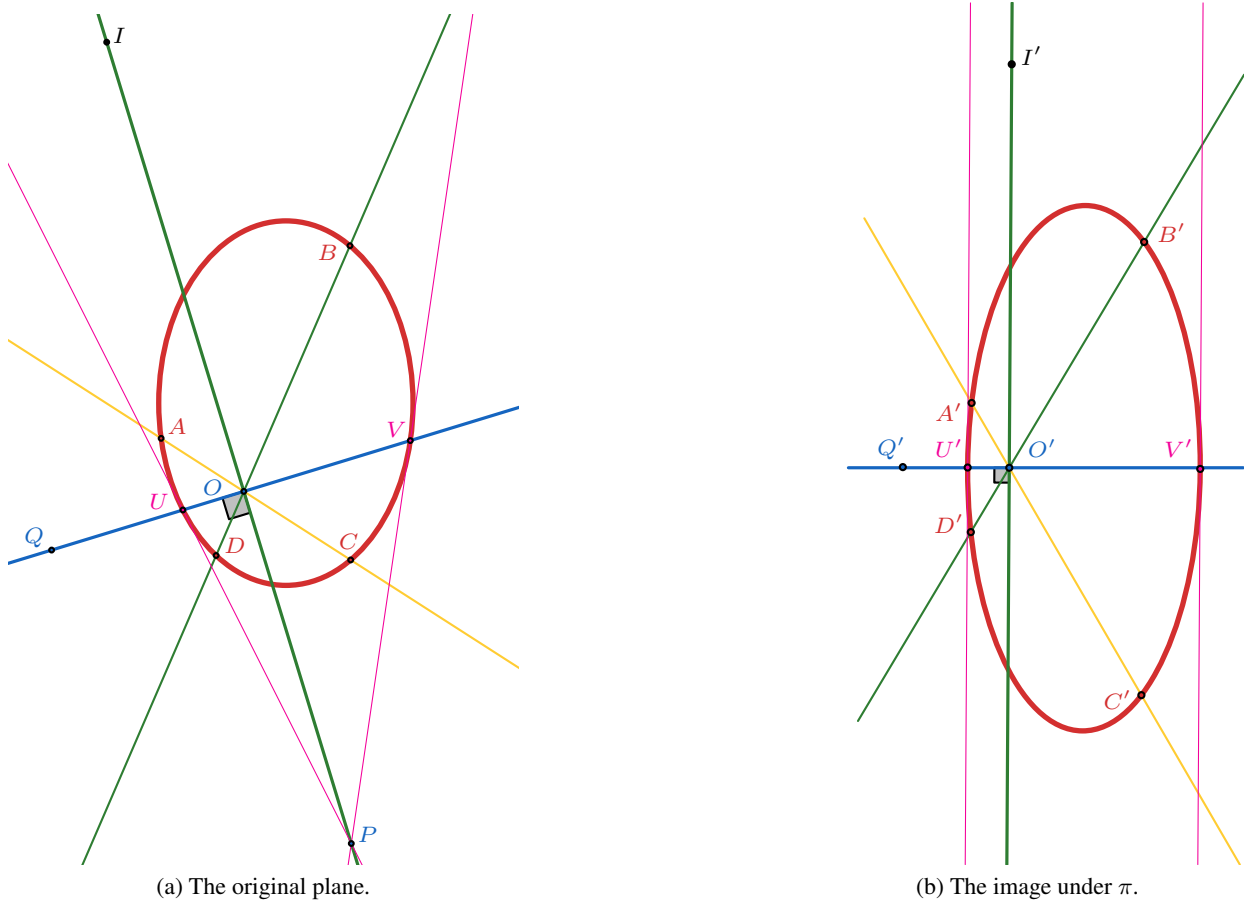

Let $P' = \infty \cap OP = [0 : 1 : 0]$, $O' = O$ and $Q' = Q$. One can always find a~projective map $\pi$ sending $O$ to $O'$, $P$ to $P'$ and $Q$ to $Q'$ (see fig.~\ref{fig:before-after}). Since a~projective transformation is defined by images of four points in general position (no~three are collinear), there is still one degree of freedom in the~choice of $\pi$. We would like the image of the ellipse $\mathfrak{a}$ to be an ellipse. To~make sure that the conic $\pi(\mathfrak{a})$ does not intersect the line at infinity, we will fix the image of another point.

There exists a line $l$, such that $P\in l$, and $l\cap\mathfrak{a}=\varnothing$. That is because $P$, as an intersection point of opposite sides $AD$ and $BC$ of the convex quadrilateral $ABCD$ inscribed in $\mathfrak{a}$, lies outside of $\mathfrak{a}$ on the Euclidean plane. Choose an arbitrary point $R \in l$, such that $R \notin OQ$. Let $R' = [0:1:1]$. Then we can find a~unique projective transformation $\pi$, such that:
$$
\pi(O) = O',\,\,\,\pi(Q) = Q',\,\, \pi(P) = P',\,\, \pi(R) = R'.
$$

Note, that $P'$ and $R'$ both lie on the line at infinity, therefore $P'R'$, the~image of $PR$ under $\pi$, is the line at infinity, and it doesn't intersect $\pi(\mathfrak{a})$.

In the next subsection we explore an important property of the map $\pi$.

\subsection{Hidden symmetry of the transformation}

We claim that $\angle O'$ of $\triangle O'P'Q'$ is right. Indeed, by the definition, $\pi$ maps $OP$ onto itself since $P' \in OP$, and $\pi$ sends $OQ$ onto itself, too. Hence $\pi$ preserves the~angle between them.

This result has a generalization:

\begin{lemma}[Hidden symmetry of $\pi$]\label{lem::2}
    Let $l_+$ and $l_-$ be two lines intersecting at $O$. Then $l_+$ and $l_-$ are reflections of each other in $OQ$ if and only if $\pi(l_+)$ and $\pi(l_-)$ are reflections of each other in $O'Q'$.
\end{lemma}

\begin{proof}
    Consider the differential of the transformation $\pi$ at the point $O$. It is a linear transformation acting on the tangent space to $\mathbb{RP}^2$ at $O$, therefore we can define its action $\Delta$ on the pencil $O^*$ of lines going through $O$ in a natural way. One can show easily that, since $\Delta$ preserves a pair of orthogonal lines $OP$ and $OQ$, it commutes with the reflections with respect to these lines in~$O^*$. Therefore, for any $l_+, l_-$ going through $O$ and symmetric with respect to $OQ$, their images $\pi(l_+), \pi(l_-)$ are symmetric with respect to $O'Q'$, and vice versa, as needed.
\end{proof}

\subsection{The transformed configuration}

From now on the prime symbol is a shorthand for application of $\pi$. For~example, $I' = \pi(I)$, $\mathfrak{a}' = \pi(\mathfrak{a})$.

Denote $U$ and $V$ the points on the ellipse $\mathfrak{a}$ at which tangent lines to the~ellipse go through $P$. There is a classical geometrical property of pole and polar: if two tangent lines can be drawn from a pole to the conic section, then its polar passes through both tangent points. Thus, the line $UV$ is the polar of $P$ with respect to the ellipse $\mathfrak{a}$. But $OQ$ is the~polar of $P$, since $\triangle OPQ$ is self-polar. Therefore we conclude that $U$ and $V$ lie on the line $OQ$.

After we apply the transformation $\pi$, the points $U'$ and $V'$ are two points where $O'Q'$ intersects $\mathfrak{a}'$. The point $P$, which is the pole of $OQ$, went to infinity, so tangents to $\mathfrak{a}'$ at $U'$ and $V'$ are parallel to $I'O'$. Consequently, they are perpendicular to $O'Q'$. Therefore, $O'Q'$ must coincide with an axis of $\mathfrak{a}'$: there are no other two-periodic orbits in ellipses~\cite{serge_2005}.

Recall that $P$ is defined as the intersection point of $AD$ and $BC$. Hence $A'D'$, $B'C'$ and $I'O'$ are all parallel to each other.
Equivalently, one can say that $A'D' \perp O'Q'$ and $B'C' \perp O'Q'$.

Now it is clear that $A'B'C'D'$ is a trapezoid, and it is symmetric with respect to $O'Q'$, since it is an axis of $\mathfrak{a}'$. Hence $O'Q'$ is the bisector of $\angle A'O'D'$. Since $O'Q'$ is perpendicular to $I'O'$, we obtain that $O'B'$ is the~reflection of $O'A'$ in $O'Q'$. But then $OB$ is the~reflection of $OA$ in $OQ$ by lemma \ref{lem::2}. This~completes the proof of the Main theorem.

\printbibliography

\end{document}